\documentclass[12 pt]{amsart}
\usepackage{amscd,amssymb,amsthm}
\usepackage[arrow,matrix]{xy}
\usepackage {graphicx}
\usepackage{csquotes}
\usepackage{mathtools}
\usepackage{tikz-cd}
\oddsidemargin=0.3in \evensidemargin=0.3in \theoremstyle{plain}
\newtheorem{thm}[subsection]{Theorem}

\newtheorem{prop}[subsection]{Proposition}
\newtheorem{cor}[subsection]{Corollary}
\newcommand{\Imm}{\operatorname{Im}}
\newcommand{\Ker}{\operatorname{Ker}}

\theoremstyle{definition}
\newtheorem{rk}[subsection]{\textmd{Remark}}
\newtheorem{definition}[subsection]{\textmd{Definition}}

\begin{document}
\title {Topology On BCK-Modules}

\author{Agha Kashif, M. Aslam} 
\address{Department of Mathematics,
University of Management and Technology , Lahore, Pakistan.\\
Email:kashif.arshad@lhr.nu.edu.pk}
\address{Department of Mathematics,
GC University, Lahore, Pakistan\\
Email: aslam298@hotmail.com}


\begin{abstract}

In this paper, we introduce the notion of a BCK-topological module in a natural way and establish that every decreasing sequence of submodules on a BCK-module $M$ over bounded commutative BCK-algebra $X$ is indeed a BCK-topological module. We have defined the notion of compatible and strict BCK-module homomorphisms, and establish that a strict BCK-module homomorphism is an open as well as a continuous mapping. Also, we establish the necessary and sufficient condition for a compatible mapping to be strict. 
\end{abstract}

\maketitle


\section{Introduction} \label{sec:intro}

A BCK-module is an action of BCK-algebras on a commutative group. It was introduced  by H.A.S. Abujabal, M. Aslam and A.B. Thaheem in \cite{AAT}. They showed in \cite{AAT} that every bounded implicative BCK-algebra forms a BCK-module over itself and developed the isomorphism theorems. Z. Perveen, M. Aslam and A.B. Thaheem introduced the notions of chains, injective and projective on BCK-modules in \cite{PAT}. I. Baig and M. Aslam in \cite{BA} introduced the notion of matrices of endomorphisms, topology over decreasing sequence of submodules of a BCK-module, BCK-module over polynomials, Artinian and Noetherian modules and discuss their properties. The theory was further explored in \cite{KA} by A. Kashif and M. Aslam. They constructed some new examples of BCK-modules in support of the theory, initiated the homology theory of BCK-modules and established that every short exact sequence of complexes admits an exact homology sequence of BCK-modules.

Let $(X,*,0)$ be a BCK-algebra and $M$ be an abelian group under addition $+$. Then $M$ is said to be an $X-$module if there exists a mapping $(a,m)\rightarrow am$ from $X \times M \rightarrow M$ satisfying the following conditions for all $a,b \in X$ and $m_1,m_2\in M$;
M1)  $(a\wedge b)m=a(bm)$, M2) $a(m_1+m_2)=am_1+am_2$, M3)  $0m=0$, where $a\wedge b=b*(b*a)$. If $X$ is bounded, then the following additional condition holds;
M4)  $1m=m$  . A right $X-$module can be defined similarly. Throughout this paper, by an
$X-$module $M$, we always mean a left BCK-module, unless stated otherwise.

A subgroup $N$ of an $X-$module $M$ is called submodule of $M$ if $N$ is also an $X-$module. Let $M_1,M_2$ be $X-$modules. A mapping $f:M_1\rightarrow M_2$ is called an $X-$homomorphism if for any $x\in X$ and $m_1,m_2\in M_1$ the following hold: 1)  $f(m_1+m_2)=f(m_1)+f(m_2)$,
2)  $f(xm_1)=xf(m_1)$. An $X-$homomorphism $f:M_1\rightarrow M_2$ which is both one to one as well as onto is called an $X-$isomorphism. The $\Ker f$ and $\Imm f$ , both in usual sense, are submodules of $M_1$ and $M_2$ respectively (see \cite{PAT}). Let $M$ be an $X-$module and $N$ be a submodule of $M$, the quotient group $M/N$ forms an $X-$module called the factor module admitted by the scalar multiplication $X\times(M/N)\rightarrow M/N$ defined by $(a,m+N)\rightarrow am+N$   $\forall a\in X, m\in M$ (see for details \cite{AAT,BA,PAT}). Some further extensions of BCK-modules can be seen in \cite{BSJ,HN}.

Here, we include some preliminaries from the theory of BCK-algebras. A BCK-algebra is an algebraic system $(X,*,0)$ that satisfies the following axioms for all $a,b,c\in X$ : BCK1)  $(a*b)*(a*c)(c*b)=0$,
BCK2)  $(a*(a*b))*b=0$, BCK3)  $a*a=0$, BCK4)  $0*a=0$, BCK5)  $a*b=0,b*a=0$ implies $a=b$ , BCK6)  $a*b=0$ iff $a\le b$. It can be observed that $(X,\le)$ forms a poset. In sequel, we denote the BCK-algebra $(X,*,0)$  simply by $X$. If $X$ contains an element $1$ such that $a\le 1$ for all $a$ in $X$, then $X$ is called bounded, $X$ is called commutative if $a\wedge b=b*(b*a)$ holds for all $a,b$ in $X$, whereas a BCK-algebra $X$ is said to be implicative if $a*(b*a)=a$ for all $a,b\in X$. We refer \cite{AT, II,IT,IT2} for undefined terms and more details of BCK-algebras.

Throughout this paper, whenever we refer $M$ as an $X-$module, we mean $M$ as module over bounded commutative BCK-algebra $X$, unless stated otherwise. For undefined terms of topological spaces see \cite{P,W}.

\section{BCK-Topological Modules}

In this section, we introduce the notion of a BCK-topological module in a natural way and establish that every Baig topology on a decreasing sequence of submodules of an $X-$module $M$ forces $M$ to be a BCK - topological module.\\
\begin{definition}\label{df-bt}\cite{BA} Let $M$ be a module over a bounded commutative BCK-algebra $X$ and
\begin{equation}
 \{M_n\vert M_n\subseteq M_{n+1}\}_{n\in \mathbb{Z}^+}
 \end{equation}
be a decreasing sequence of submodules ({\em dss}) of $M$. Then the collection
\begin{equation}\label{eq-bt}
\Im =\{V\subseteq M\vert \forall\; v\in V\exists\; n\in \mathbb{Z}^+\; \mbox{such that}\; v+M_n\subseteq V\}
\end{equation}
of subsets of $M$ forms a topology on $M$. This topology $\Im$  is referred as a Baig\footnote{ \framebox{Imran Baig} and M. Aslam defined in \cite{BA}, the topology on a decreasing sequence of submodules on $M$.} topology on $M$ admitted by the decreasing sequence (1) of its submodules.
\end{definition}

It was proved in \cite{BA} that every submodule $N$ of $M$ is both open and closed in $\Im$. Also it is well known from the study of topology that if a subset $A$ of a topology is both open and closed, then the characteristic function $\chi_A$ associated with such a subset is continuous (see \cite{P}). This leads to the following immediate corollary.

\begin{cor}
{\em Let $M$ be an $X-$module and $\Im$ be the Baig topology on $M$ admitted by the {\em dss} (1). If $N$ is a submodule of $M$, then the characteristic function $\chi_N$ is continuous in the Baig topology.}
\end{cor}

\begin{rk}
A topological space $X$ is connected iff $X$ and $\phi$ are its only open sets which are closed also (see \cite{W}). It was proved in \cite{BA} that a proper submodule $N$ of $M$ in the Baig topology $\Im$ is open iff it is closed. Indeed, one can conclude that in such case the Baig topology $\Im$ is not connected.
\end{rk}

\begin{prop}
{\em Let $M$ be an $X-$module and $\Im$ be the Baig topology on $M$ admitted by the {\em dss} (1). Then for any $m$ in $M$ the mappings,\\
\begin{description}
  \item[(i)] $\nu :M\to M$ defined by $\nu (m)=-m $
  \item[(ii)] for $a\in M, \tau_a:M\to M$ defined by $\tau_a(m)=a+m,$
\end{description}
is a homeomorphisms of $M$ onto itself.}
\end{prop}

\begin{proof}
   (i) It was shown in \cite{BA} that  $\mathfrak{B}=\{x+M_n :x\in M,n\in \mathbb{Z}^+\}$ forms a base for $\Im$. Indeed for each $U=x+M_n\in \mathfrak{B}$, there exists $V=-x+M_n$ such that $ \nu(V)=\nu(-x+M_n)=\nu(-x)+M_n=x+M_n=\textit{U}$. This shows that $\nu$  is continuous. Also from the definition, $\nu$ is  one to one, onto and satisfies $\nu^2(m)=\nu(-m)=m$ for all $m\in M$ which implies that $\nu={\nu}^{-1}$. This shows $\nu$ is also continuous and hence it is a homeomorphism.
 \\(ii) Next, we show that $\tau_a$ is a homeomorphism. Let $y=\tau_a(m)$ for some $m\in M$  and $U=y+M_n=\tau_a(m)+M_n \in \mathfrak{B}$ containing $y$. Then there exists $V=(-a+m)+M_n \in \mathfrak{B}$  such that $m\in V$  and $\tau_a(V)=\tau_a(a+m)=-a+(a+m)+M_n=m+M_n=U$. This shows that $\tau_a$ is continuous. It is clear from the definition of $\tau_a$ , it is surjective as well as injective. Now, for each  $a\in M,\tau_{-a}\tau_a(m)=\tau_{-a}(a+m)=-a+a+m=m$ for all $m\in M$. Similarly, $\tau_a\tau_{-a}(m)=m$ for all $m\in M$. Thus $\tau_a\tau_{-a}=\tau_{-a}\tau_a=I \Rightarrow \tau^{-1}_a=\tau_{-a}$. The continuity of $\tau_a$ implies the continuity of $\tau_{-a}$. Consequently, $\tau_a$ is an homeomorphism.
\end{proof}
\begin{prop}
{\em Let $M$ be an $X-$module and $\Im$ be the Baig topology on $M$ admitted by its {\em dss} (1) and suppose $M\times M$ be endowed with the product topology. Then the mapping $f:M\times M\to M$ defined by $f(m,m^\prime)=m+m^\prime$ for $m,m^\prime \in M $ is continuous.}
\end{prop}
\begin{proof}
Let $\mathfrak{B}=\{x+M_n:x\in M,n\in {\mathbb{Z}}^+\}$ be the base for $\Im$(see \cite{BA}) and consider a basic open set $a+M_n \in \mathfrak{B}$ and its inverse image $f^{-1}(a+M_n)$ in $M\times M$. If $f^{-1}(a+M_n)=\phi$, then result holds trivially. For nontrivial case, let
$f^{-1}(a+M_n)\neq\phi$ and $(x,y)\in f^{-1}(a+M_n)$. Therefore, $f(x,y)\in (a+M_n)\Rightarrow f(x,y)+M_n=a+M_n\Rightarrow (x+y)+M_n=a+M_n$. This implies$$ f^{-1}(a+M_n)=f^{-1}((x+y)+M_n).$$ \noindent It is easy to see that, for basic open sets $\left(x+M_{n} \right),\left(y+M_{n} \right)\in \mathfrak{B} $\\
\[\left(x+M_{n} \right)\times \left(y+M_{n} \right)\subseteq f^{-1} \left(x+y+M_{n} \right)\]
%
%
This implies $\left(x,y\right)\in  f^{-1} \left(a+M_{n} \right)$.
Hence $f^{-1} \left(a+M_{n} \right)$ is open. Consequently, $f$ is continuous. This completes the proof.
%
\end{proof}
\begin{prop}
{\em Let $M$ be an $X-$module and $\Im$ be the Baig topology on $M$ admitted by a {\em dss} (1) and suppose for $x\in X, \mu_x:M \to M$ be a mapping defined by
\begin{equation}
\mu_x(m)=xm.
\end{equation}
Then $\mu_x$ is continuous.}
\end{prop}
\begin{proof}
%
Let $y=\mu _{x} \left(m\right)\in M\;\; ;m\in M$ and $U=y+M_{n} \in \mathfrak{B}$ be a basic open set containing $y$. Then there exists $V\in \mathfrak{B}$ containing $m$ and satisfying $\mu _{x} \left(V\right)\subseteq U$. Indeed, for $U=y+M_{n} =xm+M_{n} \in \mathfrak{B}$ there exists $V=m+M_{n} \in \mathfrak{B}$ such that $\mu _{x} \left(V\right)=\mu _{x} \left(m+M_{n} \right)$.

\noindent Now for  $a\in \mu _{x} \left(m+M_{n} \right)$ it is easy to see that $a\in xm+M_{n} \, =\mu _{x} \left(m\right)+M_{n}$. This implies $\mu _{x} \left(m+M_{n}\right)\subseteq \mu _{x} \left(m\right)+M_{n}$.
Similarly,  $\mu _{x} \left(m\right)+M_{n} \subseteq \mu _{x} \left(m+M_{n} \right)$. This implies $\mu _{x} \left(m\right)+M_{n} =\mu _{x} \left(m+M_{n} \right)$. Therefore, $\mu _{x} \left(V\right)=xm+M_{n} \in \mathfrak{B}$ and hence $\mu _{x} \left(V\right)\subseteq U$.
This shows $\mu _{x} $ is continuous.
\end{proof}
\begin{rk}
From Proposition 2.5 and Proposition 2.6, it follows that the addition \enquote {+} and scalar multiplication are continuous under the Baig topology $\Im$. Therefore, a BCK-module is indeed a topological module under the Baig topology. This motivates us to define the following concept.
\end{rk}
\begin{definition}
An $X-$module $M$ over a bounded commutative BCK-algebra $X$ is said to be a BCK-topological module, if \enquote{+} and scalar multiplication are continuous under some suitable topology on $M$.
\end{definition}
\begin{rk}
Every Baig topology $\Im$ on a {\em dss} of an $X-$module $M$ over bounded commutative BCK-algebra $X$ force $M$ to be a BCK-topological module.
\end{rk}
It is well known from the study of normed spaces that a linear mapping in a normed space is continuous if and only if it is continuous at its $0$ element. The following proposition has the same spirit of motivation.
\begin{prop}\label{prp0}{\em
Let $M$ and $M^\prime$ be $X-$modules and $\Im$ and $\Im^\prime$ be the respective Baig topologies on them admitted by {\em dss} $\{M_n\vert M_n\subseteq M_{n+1}\}_{n\in \mathbb{Z}^+}$ and $\{M^\prime_n\vert M^\prime_n\subseteq M^\prime_{n+1}\}_{n\in \mathbb{Z}^+}$ respectively. Suppose $f:M\to M^\prime$ be an $X-$homomorphism. Then $f$ is continuous on $M$ if and only if it is continuous at the zero element of $M$.}
\end {prop}
\begin{proof}
The continuity of $f$ on $M$ trivially implies that it is continuous at zero. Conversely, assume that $f$ is continuous on zero, we will show that it is continuous at each $m\in M$.
\\Since $f$ is continuous at $0\in M$, therefore for arbitrary basic open set $M'_{k} \subseteq M'$ such that  $f\left(0\right)\in M'_{k}$, there exists a basic open set $M_{n}$ in $M'$ containing $0$ such that
\begin{equation}\label{eq1}
f\left(M_{n} \right)\subseteq M'_{k}.\end{equation}
Let $y=f\left(m\right)$ be an arbitrary element in $M'_{n}$ and $U=y+M'_{k} $ be a basic open set containing $y$.
Then there exists $V=m+M_{n} \in \mathfrak{B}$, such that the $X-$homomorphism of $f$ and eq. \ref{eq1} implies  $f\left(V\right)=f\left(m\right)+f\left(M_{n} \right)\subseteq f\left(m\right)+M'_{k} =U.$ Shows $f\left(V\right)\subseteq U$ and hence $f$ is continuous.
\end{proof}

\section{Compatible and Strict BCK-Module Homomorphisms}
In this section, we introduce the notions of compatible and strict BCK-module homomorphisms and explore their various features. We will show that every strict mapping is open as well as continuous. The notions of factor Baig topology and induced Baig topology will be introduced and explored. Finally, We will establish the necessary and sufficient condition for a compatible mapping to be strict.
\begin{definition}\label{df-cs}
Let  $M$ and $M^\prime$ be $X-$modules and $\Im$ and $\Im^\prime$ be the respective Baig topologies on them admitted by {\em dss} $\{M_n\vert M_n\subseteq M_{n+1}\}_{n\in \mathbb{Z}^+}$ and $\{M^\prime_n\vert M^\prime_n\subseteq M^\prime_{n+1}\}_{n\in \mathbb{Z}^+}$  respectively. Then an $X-$homomorphisms $f:M\to M^\prime$ is said to be \textbf{compatible} with the Baig topologies if
\begin{equation}\label{cmpt}
f(M_n)\subseteq M^\prime_n\;\;\forall n\in \mathbb{Z}^+
\end{equation}
and $f$ is said to be \textbf{strict} if
\begin{equation}\label{strc}
f(M_n)=f(M)\cap M^\prime_n\;\;\forall n\in \mathbb{Z}^+
\end{equation}
\end{definition}
\begin{prop}\label{st-cp}{\em
Let $\Im$ and $\Im^\prime$ be Baig topologies on $X-$modules $M$ and $M^\prime$ admitted by {\em dss} $\{M_n\vert M_n\subseteq M_{n+1}\}_{n\in \mathbb{Z}^+}$ and $\{M^\prime_n\vert M^\prime_n\subseteq M^\prime_{n+1}\}_{n\in \mathbb{Z}^+}$  respectively. Then an $X-$homomorphisms $f:M\to M^\prime$ is compatible if it is strict.}
\end{prop}
\begin{proof}
Let $f:M\to M'$ be strict. Then eq. (\ref{strc}) implies $f\left(M_{n} \right)=f\left(M\right)\cap M'_{n}\subseteq M'_{n}\,\; \forall\; n\in {\mathbb Z}^{+}$. Shows that $f$ is compatible.

%
%
\end{proof}
\begin{prop}\label{prp3}{\em
Let $M$ and $M^\prime$ be $X-$modules and $f:M\to M^\prime$ be an $X-$homomorphism. If $K$ is a submodule of $M$, then
$$f(K+m)=f(K)+f(m)\;\;\forall m\in M$$
}\end{prop}
\begin{proof}
%
\noindent Let $m^\prime \in f(K+m)\Rightarrow m^\prime=f(k+m);k \in K$
$\Rightarrow m^\prime =f(k)+f(m)\Rightarrow m^\prime \in f(K)+f(m)\Rightarrow f(K+m)\subseteq f(K)+ f(m)$.
				
\noindent For reverse inclusion let $m^\prime \in f(K)+f(m)\Rightarrow m^\prime =f(k)+f(m);k \in K,\Rightarrow m^\prime=f(k+m)\Rightarrow m^\prime \in f(K+m)\Rightarrow f(K)+ f(m)\subseteq f(K+m) $
\\Hence $f(K+m)=f(K)+ f(m) \forall m \in M $ .
\end{proof}
\begin{prop}\label{st-op}{\em
Let $\Im$ and $\Im^\prime$ be Baig topologies on $X-$modules $M$ and $M^\prime$ admitted by {\em dss} $\{M_n\vert M_n\subseteq M_{n+1}\}_{n\in \mathbb{Z}^+}$ and $\{M^\prime_n\vert M^\prime_n\subseteq M^\prime_{n+1}\}_{n\in \mathbb{Z}^+}$  respectively. Then every strict $X-$homomorphisms $f:M\to M^\prime$ is an open mapping.
}\end{prop}
\begin{proof}
In order to show that $f$ is an open mapping it is sufficient to show for an arbitrary basic open set $m+M_{n} $ in $M$, $f\left(m+M_{n} \right)$ is an open set. From Proposition \ref{prp3} and the strictness of $f$, one arrives at $f\left(M_{n} +m\right)=\left(f\left(M\right)\bigcap M'_{n} \right)+f\left(m\right)$. Next we show that
%
$\left(f\left(M\right)\bigcap M'_{n} \right)+f\left(m\right)=f\left(M\right)\bigcap \left(M'+f\left(m\right)\right)$.

\noindent Let  $y\in \left(f\left(M\right)\cap M'_{n} \right)+f\left(m\right)$. Then there exists $m_1\in M$ such that $y=f\left(m_{1} \right)+f\left(m\right)=f\left(m_{1} +m\right)$. This implies $y\in f\left(M\right)$. Also it is evident from above that $y\in M_{n}^{'} +f\left(m\right)$. This implies $ y\in f\left(M\right)\cap \left(M_{n} ^{'} +f\left(m\right)\right)$ and hence $\left(f\left(M\right)\cap M_{n} ^{{'} } \right)+f\left(m\right)\subseteq f\left(M\right)\cap \left(M_{n} ^{{'} } +f\left(m\right)\right)$.
\\The reverse inclusion can be proved in a similar manner. Hence
$$\left(f\left(M\right)\cap M'_{n} \right)+f\left(m\right)=f\left(M\right)\cap \left(M'_{n} +f\left(m\right)\right)$$
%

\noindent Now, $f\left(M\right)$ is open in $M'$, since $f^{-1} \left(f\left(M\right)\right)=M$. Also $M'_{n} +f\left(m\right)$ is basic open set in $M'$, therefore, $f\left(M\right)\cap \left(M'_{n} +f\left(m\right)\right)$ is an open set being intersection of open sets.

Consequently, $f\left(M_{n} +m\right)$ is open. This completes the proof.

%
%
%
%
%
\end{proof}

\begin{prop}\label{cp-cnt}{\em
Let $\Im$ and $\Im^\prime$ be Baig topologies on $X-$modules $M$ and $M^\prime$ admitted by {\em dss} $\{M_n\vert M_n\subseteq M_{n+1}\}_{n\in \mathbb{Z}^+}$ and $\{M^\prime_n\vert M^\prime_n\subseteq M^\prime_{n+1}\}_{n\in \mathbb{Z}^+}$  respectively. If $f:M\to M^\prime$ is compatible with Baig topologies $\Im$ and $\Im^\prime$, then it is continuous everywhere.
}\end{prop}
\begin{proof}
The compatibility of $f:M\to M'$ implies $f\left(M_{n} \right)\subseteq M'_{n} \;\;\forall n\in {\mathbb Z}^{+}$. In order to show that $f$ is continuous everywhere, it is sufficient to show its continuity at zero.
%

\noindent Let $M'_{n} $ be an arbitrary basic open set containing $f\left(0\right)$. Then using compatibility of $f$, it follows that $0\in M_{n} \subseteq f^{-1} \left(M'_{n} \right)$. This shows $f$ is continuous at $0$ and hence by Proposition \ref{prp0} $f$ is continuous everywhere. This completes the proof.
%
%
%
%
%
%
\end{proof}
The following are immediate corollaries of the ongoing discussions.
\begin{cor}\label{st-cnt}{\em
Every strict mapping is continuous.}
\end{cor}
\begin{proof}
The following diagram explains the proof.
$$Strictness\stackrel{Proposition\; \ref{st-cp}}{\Longrightarrow} Compatiblity \stackrel{Proposition\; \ref{cp-cnt}}{\Longrightarrow} Continuity$$
\end{proof}

\begin{cor}\label{stb-hom}{\em
If $f:M\to M^\prime$ is an $X-$module isomorphism, then $f$ is an homeomorphism provided it is strict.}
\end{cor}
\begin{proof}
Since $f$ is strict, therefore by Corollary \ref{st-cnt} $f$ is continuous. Also by Proposition \ref{st-op}, it follows that $f$ is an open mapping. Now $f$ and $f^{-1} $ are continuous and also \textit{f} is bijective.  This implies that  $f$ is an homeomorphism.
\end{proof}

\begin{prop}\label{prp-ibt}{\em
Let $M$ be an $X-$module and $\Im$ be a Baig topology on $M$ admitted by its {\em dss} $\{M_n\vert M_n\subseteq M_{n+1}\}_{n\in \mathbb{Z}^+}$. Then for any submodule $K$ of $M$ the sequence
\begin{equation}                            
\{K_n\vert K_n=K\cap M_n\}_{n\in \mathbb{Z}^+}
\end{equation} 				
admits a topology on $K$.
}\end{prop}
\begin{proof}
Since $K$ and $M_{n} $ are both submodules of $M\;\forall n$, therefore, $K_{n} =K\cap M_{n}\;\forall n$ is also a submodule of $M$ such that $K_{n} \supseteq K_{n+1}\;\forall n $. Indeed, the set $\left\{K_{n} \left|K_{n} =K\cap M_{n} \right. \right\}_{n\in {\mathbb Z}^{+} } $forms a {\em dss} of $K$. Let $\Im _{i} $ be a collection of subsets of $K$ defined as
\begin{equation} \label{eq-ibt}
\Im _{i} =\left\{U\subseteq K\left|u\in U\, \, \exists n\in {\mathbb Z}^{+} \, {\rm such\; that\; }u+K_{n} \subseteq U\right. \right\}.
\end{equation}
Then by Definition \ref{df-bt} $\Im _{i} $ forms a Baig topology on $K$.
%
%
%
\end{proof}
The topology $\Im _{i} $ introduced in eq. (\ref{eq-ibt}) is referred as Induced Baig topology on $K$. This leads to the following definition.
\begin{definition}\label{df-ibt}
Let $M$ be an $X-$module and $\Im$ be a Baig topology on $M$ admitted by its {\em dss} $\{M_n\vert M_n\subseteq M_{n+1}\}_{n\in \mathbb{Z}^+}$. Then for every submodule $K$ of $M$ with respective {\em dss} $\{K_n\vert K_n=K\cap M_n\}_{n\in \mathbb{Z}^+}$, the collection $\Im_i =\{U\subseteq K\vert \forall u\in U\exists n\in \mathbb{Z}^+ \;such\; that\; u+K_n\subseteq U\}$ of subsets of $K$ forms a topology on $K$ called the induced Baig topology on $K$.
\end{definition}
From the construction of the induced Baig topology on $K$, it follows that the topology has motivated its spirit from the usual relative topology on a subset. The following proposition explains this point.

\begin{prop}{\em
Let $\Im$ be a Baig topology on an $X-$module $M$ admitted by its {\em dss} $\{M_n\vert M_n\subseteq M_{n+1}\}_{n\in \mathbb{Z}^+}$. Then for any submodule $K$ of $M$ the Induced Baig topology $\Im_i$, is a relative topology on $K$ w.r.t. the Baig topology $\Im$.
}\end{prop}
\begin{proof}
\noindent Let $l\in K$ be an arbitrary element and consider $y\in l+K\cap M_{n}$. Then there exists some $m\in K\cap M_{n} $ such $y=l+m$. This implies $y\in K$. Also $m\in M_{n}$ implies $y=l+m\in l+M_{n} $. This shows that $y\in \left(l+M_{n} \right)\cap K$ and hence $l+K\cap M_{n} \subseteq \left(l+M_{n} \right)\cap K$. The reverse inclusion can be proved similarly. Therefore, one obtains the following:
$$\left(l+M_{n} \right)\cap K=l+K\cap M_{n} =l+K_{n} \;\; \because K_{i} =l+K_{i} \in \Im _{i}$$
%
%
%
Since $l+M_{n} $ is an open set in $\Im_i $, therefore  $\left(l+M_{n} \right)\cap K$ is open in relative topology on $K$. Therefore, basic open sets of  $\Im _{i} $ are open set of the relative topology on \textit{K}. Consequently, $\Im _{i} $ is a relative topology on $K$ w.r.t. the Baig topology $\Im $.
%
%
%
%
%
%
%
%
\end{proof}
\begin{prop}{\em
Let $\Im$ be a Baig topology on an $X-$module $M$ admitted by its {\em dss} $\{M_n\vert M_n\subseteq M_{n+1}\}_{n\in \mathbb{Z}^+}$. Then for any submodule $K$ of $M$ the sequence
\begin{equation}
\Big\{ \frac{M_n+K}{K}\vert n\in \mathbb{Z}^+\Big\}
\end{equation}
admits a topology on the factor $X-$module $\frac{M}{K}$.
}\end{prop}
\begin{proof}
Since $M_{n} +K$ is a submodule of $M$, therefore, $\frac{M_{n} +K}{K} $ is a submodule of the factor module $\frac{M}{K} $. Also $M_{n} \supseteq M_{n+1} \;\forall n\in {\mathbb Z}^{+} $, therefore, $ M_{n} +K\supseteq M_{n+1} +K\; \forall n\in {\mathbb Z}^{+} $. This implies $ \frac{M_{n} +K}{K} \supseteq \frac{M_{n+1} +K}{K} \; \forall n\in {\mathbb Z}^{+}$. Therefore, $\left\{\frac{M_{n} +K}{K} \left|n\in {\mathbb Z}^{+} \right. \right\}$ forms a {\em dss} of the factor $X-$module $\frac{M}{K} $. Let $\Im _{f} $ defines a collection of subsets of $\frac{M}{K} $ as below:
\begin{equation} \label{eq-fbt}
\Im _{f} =\left\{\frac{N}{K} \subseteq \frac{M}{K} \left|\;\forall \bar{U}\in \frac{N}{K}\; \exists\; \, n\in {\mathbb Z}^{+} \, \, {\rm such\; that\; \; }\bar{U}+\frac{M_{n} +K}{K} \subseteq \frac{N}{K} \right. \right\}.
\end{equation}
Then by Definition \ref{df-bt} $\Im _{f} $ forms a topology on $\frac{M}{K} $.
%
%
\end{proof}
The above proposition leads to the following definition.
\begin{definition}
Let $M$ be an $X-$module and $\Im$ be a Baig topology on $M$ admitted by its {\em dss} $\{M_n\vert M_n\subseteq M_{n+1}\}_{n\in \mathbb{Z}^+}$. Then for any submodule $K$ of $M$  and the {\em dss} $\{ \frac{M_n+K}{K}\vert n\in \mathbb{Z}^+\}$ of its factor module $\frac{M}{K}$, the collection of subsets $\Im_f =\{U\subseteq K\vert \;\forall \;u\in U\;\exists\; n\in \mathbb{Z}^+ \mbox{such that}\; u+K_n\subseteq U\}$ admits a topology on $\frac{M}{K}$ called the factor Baig topology on $M$ w.r.t. $K$.
\end{definition}

\begin{prop}{\em
Let $M$ be an $X-$module. Then for any submodule $K$  and respective factor $X-$module $M/K$ following hold:
\begin{description}
\item[a)] The sequence
\begin{equation}\label{esq}
K\stackrel{i}\longrightarrow M \stackrel{f}\longrightarrow M/K
\end{equation}
where $i(k)=k;\;k\in K$ is the inclusion mapping and $f(m)=m+K;\;m\in M$ is the natural $X-$homomorphism, is an exact sequence.
\item [b)] The inclusion mapping $i$ and the natural $X-$homomorphism $f$ are strict if $K,M$ and $M/K$ are equipped with their respective Baig topologies $\Im_i,\Im$ and $\Im_f$ (c.f. eqs. (\ref{eq-bt},\ref{eq-ibt},\ref{eq-fbt})).
\end{description}
}\end{prop}
\begin{proof} \textbf{a):} The inclusion mapping $i$ and the mapping $f$ are indeed natural homomorphism such that $\Imm i=K=\Ker f$. Therefore, diagram (\ref{esq}) forms an exact sequence.

\noindent\textbf{b):} Since $K_n=K\cap M_n$ by Definition \ref{df-ibt}, therefore, $i\left(K_{n} \right)=i\left(K\cap M_{n} \right)\, =K\cap M_{n} =i\left(K\right)\cap M_{n}$. Hence by eq. (\ref{strc}), $i$ is strict.

 Next, we will show that $f$ is strict. Since $f\left(M\right)=\frac{M}{K}$ and the fact that $M_{n} +K\subseteq M$ therefore, $f\left(M\right)\cap \frac{M_{n} +K}{K} =\frac{M_{n} +K}{K} $. Also it is easy to see that $f\left(M_{n} \right)=\frac{M_{n} +K}{K}$. Hence $f\left(M\right)\cap \frac{M_{n} +K}{K} =f\left(M_{n} \right)$. Therefore by eq. (\ref{strc}), $f$ is strict.
%
%
%
%
%

\end{proof}

\begin{cor}{\em
$i$ and $f$ are open mappings.}
\end{cor}
\begin{proof}
An immediate implication of Proposition \ref{st-op}.
\end{proof}

\begin{thm}\label{thm1}{\em
Let $\Im$ and $\Im^ \prime$ be Baig topologies on $X-$modules $M$ and $M^ \prime$ admitted by {\em dss} $\{M_n\vert M_n\subseteq M_{n+1}\}_{n\in \mathbb{Z}^+}$ and $\{M^\prime_n\vert M^\prime_n\subseteq M^\prime_{n+1}\}_{n\in \mathbb{Z}^+}$  respectively. If  $f:M\to M^\prime$ is a compatible $X-$homomorphism, then for all $n$, $f$ induces and $X-$homomorphism $f_n :\frac{M}{M_n}\to \frac{M^\prime}{M^\prime_n}$ such that the diagram
\begin{equation}\label{d1}
\begin{tikzcd}
M\arrow{r}{f}\arrow{d}{\phi_n}
&M^\prime\arrow{d}{\phi^\prime_n}\\
{\frac{M}{M_n}}\arrow{r}{f_n}&{\frac{M^\prime}{M^\prime_n}}
\end{tikzcd}
\end{equation}
commutes, where $\phi_n$ and $\phi^\prime_n$ are natural $X-$homomorphisms from $M$ onto $M/M_n$ and $M^\prime$ onto $M^\prime/M^\prime_n$ respectively.
}\end{thm}
\begin{proof}
Let $f_n :\frac{M}{M_n}\to \frac{M^\prime}{M^\prime_n}$ be defined in a natural way as $f_{n} \left(m+M_{n} \right)=f\left(m\right)+M'_{n} \; ;m\in M$. First we will establish that $f_n$ is well defined. \\Let   $\bar{m}_{1} =m_{1} +M_{n} {\rm ,\; }\bar{m}_{2} =m_{2} +M_{n} \in \frac{M}{M_n}$ such that $\bar{m}_{1} =\bar{m}_{2} $. Then $m_{1} -m_{2} \in M_{n}$ and therefore, $f\left(m_{1} -m_{2} \right)\in f\left(M_{n} \right)\subseteq M'_{n}$ since $f$ is compatible. The $X-$homomrophism of $f$ implies that $f\left(m_{1} \right)+M'=f\left(m_{2} \right)+M'$ and hence $f_{n} \left(\bar{m}_{1} \right)=f_{n} \left(\bar{m}_{2} \right)$. This shows that $f$ is well defined. Also the natural induction of $f_n$ from the $X-$ homomorphism $f$ implies that it is also an $X-$homomorphim.

\noindent The only thing remains to establish is that diagram (\ref{d1}) commutes, i.e., $\phi '_{n} f=f_{n} \phi _{n} $.
\noindent Let $m\in M$. Then
$f_{n} \phi _{n} \left(m\right)=f_{n} \left(\phi _{n} \left(m\right)\right)$. Also $\phi _{n} \left(m\right)=m+M_{n}$ implies that $f_{n} \phi _{n} \left(m\right)=f_{n} \left(m+M_{n} \right)=f\left(m\right)+M'_{n}=\phi '_{n} \left(f\left(m\right)\right)=\phi '_{n} f\left(m\right)\;\;\forall m\in M$. Hence $f_{n} \phi _{n} =\phi '_{n} f$. This completes the proof.
\end{proof}

\begin{cor}\label{crthm1}{\em
Consider the notation used in the Theorem \ref{thm1} and the diagram (\ref{d1}).
Then for each $n\in \mathbb{Z}^+$, there exists an $X-$homomorphism $\alpha_n:Kerf \to Kerf_n$.
}\end{cor}
\begin{proof}
Let $k\in \Ker f$. Then $f\left(k\right)=0\in M'_{n} \; \forall n$ since $M'_{n} $ are submodules of $M'$ for all $n$. This implies $\phi '_{n} \left(f\left(k\right)\right)=\phi '_{n} f\left(k\right)=M'_{n}$ is the zero of $\frac{M'} {M'_{n} }$. Since the diagram (\ref{d1}) is commutative, therefore, $f_{n} \phi _{n} \left(k\right)=\phi'_{n} f\left(k\right)=M'_{n}$. This implies that $\phi _{n} \left(k\right)\in \Ker f_{n}$.
Let $\alpha _{n} :\Ker f\to \Ker f_{n} $  be defined as:
$$\alpha _{n} \left(k\right)=\phi _{n} \left(k\right)$$

\noindent Indeed, $\alpha _{n} $ is a well defined $X-$homomorphism, since $\phi _{n} $ is a well defined $X-$hom-omorphism. This completes the proof.

\end{proof}
The following theorem presents the necessary and sufficient condition for a compatible mapping to be strict.

\begin{thm}\label{thm2}{\em
Let $\Im$ and $\Im^ \prime$ be Baig topologies on $X-$modules $M$ and $M^ \prime$ admitted by {\em dss} $\{M_n\vert M_n\subseteq M_{n+1}\}_{n\in \mathbb{Z}^+}$ and $\{M^\prime_n\vert M^\prime_n\subseteq M^\prime_{n+1}\}_{n\in \mathbb {Z}^+}$  respectively. Then a compatible $X-$homomorphism $f:M\to M^\prime$ is strict if and only if for all $n\in{\mathbb Z}^+$ the associated mapping $\alpha_n:\Ker f \to \Ker f_n$ is an $X-$epimorphism.
}\end{thm}
\begin{proof}
Let $f:M\to M'$ be strict. Then
$f\left(M_{n} \right)=f\left(M\right)\cap M'_{n} \; \forall n\in {\mathbb Z}^{+} $. We will shwo that $\alpha_n$ is an epimorphism. Let $k'\in \Ker f_{n}$. Then $ f_{n} \left(k'\right)=0\; ;k'\in M $. Since $\phi _{n}$ is surjective, therefore, there exists some $k\in M$ such that $f_{n} \phi _n \left(k\right)=0$. Therefore, from Theorem \ref{thm1} $\phi '_{{\rm n}} f\left(k\right)=M'_{n}$ and hence $f\left(k\right)\in \Ker\phi '_n =M'_{n}$. This implies $f\left(k\right)\in f\left(M\right)\cap M'_{n} =f\left(M_{n} \right)$. Implies there exists $ m_{n} \in M_{n}$ such that $f\left(k\right)=f\left(m_{n} \right)$. Therefore, $k-m_{n} \in Kerf$ and from Corollary \ref{crthm1}, simple calculations reveal that $\phi _{n} \left(k-m_{n} \right)=\phi _{n} \left(k\right)=k'$. Since $k'$ was arbitrary, therefore, $\alpha_n$ is an $X-$epimorphism.

Conversely, assume that $\alpha _{n} :\Ker f\to \Ker f_{n} $be an $X-$epimorphism, we will show that $f$ is strict. Indeed, the compatibility of $f$ implies $f\left(M_{n} \right)\subseteq f\left(M\right)\cap M'_{n}$. For reverse inclusion,
 let  $x'\in f\left(M\right)\cap M'_{n} $. Then there exists $x\in M$ such that $f\left(x\right)=x'$. Now the commutativity of diagram \ref{d1} implies that $\phi _{n} \left(x\right)\in Kerf_{n}$.
Since $\alpha _{n} :\Ker f\to \Ker f_{n} $ is an $X-$epimorphism, therefore, there exists $ y\in \Ker f$ such that
$\alpha _{n} \left(y\right)=\phi _{n} \left(x\right)$. Now from Corollary \ref{crthm1}, it implies that $\alpha _{n} \left(y\right)=\phi _{n} \left(y\right)$. This further implies that $x-y\in \Ker\phi _{n} =M_{n}$. Also $f\left(x-y\right)=f\left(x\right)-f\left(y\right)=f\left(x\right)=x'\; \because y\in \Ker f$. This implies $f\left(x-y\right)=x'\in f\left(M_{n} \right)\; \because x-y\in M_{n}$. Hence, $f\left(M\right)\cap M'_{n} \subseteq f\left(M_{n} \right)$ Consequently, $f\left(M\right)\cap M'_{n} = f\left(M_{n} \right)$ by Definition \ref{df-cs}. This completes the proof.
%
\end{proof}

\begin{cor}{\em
If in the last proposition $\alpha_n:\Ker f \to \Ker f_n$ is an $X-$epimorphism, then $f$ is a continuous as well as an open mapping.
}\end{cor}
\begin{proof}
 Indeed, in this case Theorem \ref{thm2} implies $f$ is strict and therefore, from Corollary \ref{st-cnt} and Proposition \ref{st-op}, it follows that $f$ is continuous and open respectively.
\end{proof}

\section{A direction to explore for BCK-modules}
The theory developed above can be further explored in different directions. One such direction is to discuss the completeness of BCK-module w.r.t. the Baig Topology. Some interesting notions and results have already been established regarding the necessary and sufficient condition of completeness of a BCK-modules. These results will be communicated after some further refinement and addition in  near future.

\end{document}